\documentclass[a4paper,12pt,intlimits,oneside]{amsart}

\usepackage{enumerate}
\usepackage{amsfonts,amsmath}

\usepackage{latexsym,amssymb}

\usepackage{mathrsfs}

\newcommand{\comment}[1]{}

\newcommand{\bN}{{\mathbb N}}
\newcommand{\bR}{{\mathbb R}}

\newcounter{rek}
\title[commutators of pseudo-differential operators]{An Hardy estimate for commutators of pseudo-differential operators}         

\author{Ha Duy HUNG}    
\address{High School for Gifted Students,
Hanoi National University of Education, 136 Xuan Thuy, Hanoi, Vietnam} 
\email{{\tt hunghaduy@gmail.com}}
\author{Luong Dang KY}
\address{Department of Mathematics, Quy Nhon University, 
170 An Duong Vuong, Quy Nhon, Binh Dinh, Viet Nam} 
\email{{\tt dangky@math.cnrs.fr}}

\thanks{The first author is supported  by Vietnam National Foundation for Science and Technology Development (NAFOSTED) under grant number 101.01-2011.42. The second author is supported  by Vietnam National Foundation for Science and Technology Development (NAFOSTED) under grant number 101.02-2014.31.}


\date{}
\begin{document}
 \maketitle

\begin{abstract}
Let $T$ be a pseudo-differential operator whose symbol belongs to the H\"ormander class $S^m_{\rho,\delta}$ with $0\leq \delta<1, 0< \rho\leq 1, \delta \leq \rho$ and $-(n+1)< m \leq - (n+1)(1-\rho)$. In present paper, we prove that if $b$ is a locally integrable function satisfying
$$\sup_{{\rm balls}\; B\subset \mathbb R^n} \frac{\log(e+ 1/|B|)}{(1+ |B|)^\theta} \frac{1}{|B|}\int_{B} \Big|f(x)- \frac{1}{|B|}\int_{B} f(y) dy\Big|dx <\infty$$
for some $\theta\in [0,\infty)$, then the commutator $[b,T]$ is bounded on the local Hardy space $h^1(\mathbb R^n)$ introduced by  Goldberg \cite{Go}. 

As a consequence, when $\rho=1$ and $m=0$, we obtain an improvement of a recent result by Yang, Wang and Chen \cite{YWC}.
\end{abstract}

\newtheorem{theorem}{Theorem}[section]
\newtheorem{lemma}{Lemma}[section]
\newtheorem{proposition}{Proposition}[section]
\newtheorem{remark}{Remark}[section]
\newtheorem{corollary}{Corollary}[section]
\newtheorem{definition}{Definition}[section]
\newtheorem{example}{Example}[section]
\numberwithin{equation}{section}
\newtheorem{Theorem}{Theorem}[section]
\newtheorem{Lemma}{Lemma}[section]
\newtheorem{Proposition}{Proposition}[section]
\newtheorem{Remark}{Remark}[section]
\newtheorem{Corollary}{Corollary}[section]
\newtheorem{Definition}{Definition}[section]
\newtheorem{Example}{Example}[section]
\newtheorem{Question}{Question}
\newtheorem*{theorema}{Theorem A}
\newtheorem*{theoremb}{Theorem B}
\newtheorem*{theoremc}{Theorem C}

\section{Introduction}

Let $T$ be a Calder\'on-Zygmund operator. A classical result of  Coifman,  Rochberg and  Weiss (see \cite{CRW}), states that the commutator $[b,T]$, defined by $[b,T](f)= bTf - T(bf)$, is continuous on $L^p(\mathbb R^n)$ for $1 <p<\infty$, when $b\in BMO(\mathbb R^n)$. Unlike the theory of Calder\'on-Zygmund operators, the proof of this result does not rely on a weak type $(1, 1)$ estimate for $[b, T]$. In fact, it was shown in \cite{Ky1, Pe} that, in general, the linear commutator fails to be of weak type $(1, 1)$ and fails to be of type $(H^1, L^1)$, when $b$ is in $BMO(\mathbb R^n)$. Instead, an endpoint theory was provided for this operator.

Let $T$ be a pseudo-differential operator which is formally defined as
$$Tf(x)= \int_{\mathbb R^n} \sigma(x,\xi) e^{2\pi i x\cdot \xi} \hat f(\xi) d\xi, \;f\in \mathcal S(\mathbb R^n),$$
where  $\hat f$ denotes the Fourier transform of $f$ and $\sigma(x,\xi)$ is a symbol in the H\"ormander class $S^{m}_{\rho,\delta}$ for some $m,\rho,\delta\in \bR$ (see Section 2). Remark that $T$ is a Calder\'on-Zygmund operator if the symbol $\sigma(x,\xi)$ satisfies some additional assumptions (cf. \cite{Jo}). In analogy with the classical results in the setting of Calder\'on-Zygmund operators, when $b\in BMO(\bR^n)$, the boundedness of $[b,T]$ on Lebesgue spaces $L^p(\bR^n), 1<p<\infty$,  have been established, see for example \cite{AT, Ch, Li, Ta}. We refer to \cite{FYY, HLY, LMY} for some similar results in the setting of metric measure spaces. It is well-known that under certain conditions of $m,\rho,\delta$, the operator $T$ is bounded on $h^1(\bR^n)$ and bounded on $bmo(\bR^n)$ (cf. \cite{Go, HK, YY1, YY2}). A natural question is that can one find functions $b$ for which $[b,T]$ is bounded on $h^1(\bR^n)$ ? Recently, some endpoint results have obtained by Yang, Wang and Chen \cite{YWC}. More precisely, in \cite{YWC}, the authors proved the following.

\begin{theorema}
Let $b\in LMO_\infty(\bR^n)$. Suppose that $T$ is a pseudo-differential operator with symbol $\sigma(x,\xi)$ in the H\"ormander class $S^0_{1,\delta}$ with $0\leq \delta<1$. Then,
\begin{enumerate}[(i)]
\item $[b,T]$ is bounded from $H^1(\bR^n)$ into $L^1(\bR^n)$.
\item $[b,T]$ is bounded from $L^\infty(\bR^n)$ into $BMO(\bR^n)$.
\end{enumerate}
\end{theorema}

Our main theorem is as follows.

\begin{Theorem}\label{the main theorem}
Let $b\in LMO_\infty(\bR^n)$. Suppose that $T$ is a pseudo-differential operator with symbol $\sigma(x,\xi)$ in the H\"ormander class $S^m_{\rho,\delta}$  with  $0\leq \delta<1, 0< \rho\leq 1, \delta \leq \rho$ and $-(n+1)< m \leq - (n+1)(1-\rho)$. Then,
\begin{enumerate}[(i)]
\item $[b,T]$ is bounded from $h^1(\bR^n)$ into itself.
\item $[b,T]$ is bounded from $bmo(\bR^n)$ into itself.
\end{enumerate}
\end{Theorem}

Throughout the whole paper, $C$ denotes a positive geometric constant which is independent of the main parameters, but may change from line to line. For any measurable set $A\subset\bR^n$, denote by $|A|$ the Lebesgue measure of $A$.

The paper is organized as follows. In Section 2, we give some notations and preliminaries about the spaces of $BMO$ type, Hardy spaces  and  pseudo-differential operators. Section 3 is devoted to prove Theorem \ref{the main theorem}. An appendix will be given in Section 4.

{\bf Acknowledgements.} The authors would like to thank  Aline Bonami and  Sandrine Grellier for many helpful suggestions and discussions. We would also like to thank the referees for their carefully reading and helpful suggestions.

\section{Some preliminaries and notations}

As usual, $\mathcal S(\mathbb R^n)$ denotes the Schwartz class of test functions on $\mathbb R^n$, $\mathcal S'(\mathbb R^n)$ the space of tempered distributions, and $C^\infty_c(\mathbb R^n)$ the space of $C^\infty$-functions with compact support.

Let $m,\rho$ and $\delta$ be real numbers. A symbol in the H\"ormander class $S^m_{\rho,\delta}$ will be a smooth function $\sigma(x,\xi)$ defined on $\mathbb R^n\times \mathbb R^n$, satisfying the estimates
$$|D_x^\alpha D_\xi^\beta \sigma(x,\xi)|\leq C_{\alpha,\beta} (1+ |\xi|)^{m-\rho |\beta| + \delta |\alpha|},\quad\alpha,\beta\in \mathbb N^n.$$

We say that an operator $T$ is a pseudo-differential operator associated with the symbol $\sigma(x,\xi)\in S^m_{\rho,\delta}$ if it can be written as
$$Tf(x)= \int_{\mathbb R^n} \sigma(x,\xi) e^{2\pi i x\cdot \xi} \hat f(\xi) d\xi, \;f\in \mathcal S(\mathbb R^n),$$
where $\hat f$ denotes the Fourier transform of $f$. Denote by $\mathscr L^m_{\rho,\delta}$ the class of pseudo-differential operators whose symbols are in $S^m_{\rho,\delta}$.

Let $0<\rho\leq 1$, $0\leq \delta<1$ and $m\in\bR$. It is well-known (see \cite[Proposition 3.1]{HK}) that if $T\in \mathscr L^m_{\rho,\delta}$ with the symbol $\sigma(x,\xi)$, then $T$ has the distribution kernel $K(x,y)$ given by 
$$K(x,y) =\lim_{\epsilon\to 0} \int_{\mathbb R^n} e^{2\pi i (x-y)\cdot \xi} \sigma(x,\xi)\psi(\epsilon \xi)d\xi,$$
where $\psi\in C^\infty_c(\mathbb R^n)$ satisfies $\psi(\xi) \equiv 1$ for $|\xi|\leq 1$, the limit is taken in $\mathcal S'(\mathbb R^n)$ and does not depend on the choice of $\psi$. 

The following useful estimates of the kernels are due to Alvarez and  Hounie \cite[Theorem 1.1]{AH}.

\begin{proposition}\label{AH, Theorem 1.1}
Let $0<\rho\leq 1$, $0\leq \delta<1$ and $T\in \mathscr L^m_{\rho,\delta}$. Then, the distribution kernel $K(x,y)$ of $T$ is smooth outside the diagonal $\{(x,x): x\in\mathbb R^n\}$. Moreover, 
\begin{enumerate}[(i)]
\item For any $\alpha,\beta\in\mathbb N^n$, $N>0$, 
$$\sup_{|x - y|\geq 1} |x-y|^N |D^{\alpha}_x D^{\beta}_y K(x,y)|\leq C(\alpha,\beta,N).$$
\item If $M\in\mathbb N$ satisfies $M+m+n >0$, then
$$\sup_{|\alpha +\beta|=M} |D^{\alpha}_x D^{\beta}_y K(x,y)|\leq C(M) \frac{1}{|x-y|^{\frac{M+m+n}{\rho}}},\quad x\ne y.$$
\end{enumerate}
\end{proposition}

Here and in what follows, for any ball $B\subset\bR^n$ and $f\in L^1_{\rm loc}(\bR^n)$, we denote
$$f_B:= \frac{1}{|B|}\int_B f(x)dx.$$

Let $0\leq \theta<\infty$. Following Bongioanni, Harboure and  Salinas \cite{BHS},  we say that a locally integrable function $f$ is in $BMO_\theta(\mathbb R^n)$, if
$$\|f\|_{BMO_\theta}:= \sup_{B}  \frac{1}{(1+ r_B)^\theta |B|}\int_{B} |f(y)- f_{B}|dy <\infty,$$
where  the supremum is taken over all balls $B \subset \mathbb R^n$. We then define 
\begin{equation}
BMO_\infty(\mathbb R^n) =\cup_{\theta \geq 0} BMO_{\theta}(\mathbb R^n).
\end{equation}

 A locally integrable function $f$ is said to belongs $LMO_\theta(\mathbb R^n)$ if
$$\|f\|_{LMO_\theta}:= \sup_{B} \frac{\log(e+ 1/r_B)}{(1+ r_B)^\theta} \frac{1}{|B|}\int_{B} |f(y)- f_{B}|dy <\infty,$$
where  the supremum is taken over all balls $B \subset \mathbb R^n$. We define 
\begin{equation}
LMO_\infty(\mathbb R^n) =\cup_{\theta \geq 0} LMO_{\theta}(\mathbb R^n).
\end{equation}

Let $\phi$ be a Schwartz function satisfying $\int_{\mathbb R^n}\phi(x) dx =1$.  According to Goldberg \cite{Go}, we define $h^1(\mathbb R^n)$ as the set of all $f\in \mathcal S'(\mathbb R^n)$ such that 
$$\|f\|_{h^1}: = \|\mathfrak m_\phi f\|_{L^1}<\infty,$$
where $\mathfrak m_\phi f(x):= \sup_{0<t\leq 1} |f*\phi_t(x)|$ with $\phi_t(x):= t^{-n}\phi(t^{-1}x)$.

Given $1<q\leq \infty$, a function $a$ is called an $(h^1,q)$-atom related to the ball $B=B(x_0,r)$ if $r\leq 2$ and 
\begin{enumerate}[(i)]
\item supp $a\subset B$,
\item $\|a\|_{L^q}\leq |B|^{1/q- 1}$,
\item if $0<r<1$, then $\int_{\mathbb R^n} a(x) dx =0$.
\end{enumerate}

The following useful fact is due to Yang and Zhou  \cite[Proposition 3.2]{YZ} (see also \cite{CCYY, YY1, YY2}).

\begin{proposition}\label{boundedness via atoms}
Let $1<q<\infty$. If $T$ is a bounded linear operator on $L^q(\mathbb R^n)$ satisfying $\|Ta\|_{h^1}\leq C $ for all $(h^1,q)$-atoms $a$, then $T$ can be extended to a bounded linear operator on $h^1(\mathbb R^n)$.
\end{proposition}

It is well-known (see \cite{Go}) that the dual space of $h^1(\mathbb R^n)$ is $bmo(\mathbb R^n)$, namely, the space of locally integrable functions $f$ such that
$$\|f\|_{bmo}:= \sup_{B\in \mathcal D} \frac{1}{|B|}\int_B |f(x) - f_B| dx + \sup_{B\in {\mathcal D}^c} \frac{1}{|B|}\int_B |f(x)| dx <\infty,$$
where  $\mathcal D=\{B(x_0,r)\subset \bR^n: 0<r<1\}$ and ${\mathcal D}^c=\{B(x_0,r)\subset \bR^n: r\geq 1\}$. 

Denote by $vmo(\mathbb R^n)$ the closure of $C^\infty_c(\mathbb R^n)$ in the space $bmo(\mathbb R^n)$. Thanks to \cite[Theorem 9]{Da}, we have the following.

\begin{theoremb}\label{Da, Theorem 9}
The dual of the space $vmo(\mathbb R^n)$ is the space $h^1(\mathbb R^n)$.
\end{theoremb}

The following result is due to Hounie and Kapp \cite[Theorem 4.1]{HK}.

\begin{theoremc}\label{HK, Theorem 4.1}
Let $T\in \mathscr L^m_{\rho,\delta}$ with $0\leq \delta<1, 0< \rho\leq 1, \delta \leq \rho$ and $m \leq - n(1-\rho)/2$. Then, $T$ is bounded on $h^1(\mathbb R^n)$.
\end{theoremc}

\section{Proof of Theorem \ref{the main theorem}}

Here and in what follows, for any ball $B=B(x_0,r)$ and $k\in\bN$, we denote 
$$2^kB:=B(x_0, 2^k r).$$

In order to prove Theorem \ref{the main theorem}, we need the following three technical lemmas.

\begin{Lemma}\label{Ky2, Lemma 5.3}
Let $1\leq q <\infty$ and $0 \leq \theta<\infty$. Then,
\begin{enumerate}[(i)]
\item There exists a constant $C= C(q,\theta)>0$ such that
$$\Big( \frac{1}{|2^k B|} \int_{2^k B} |f(y) - f_B|^q \Big)^{1/q}\leq C k (1+ 2^k r)^{2\theta} \|f\|_{BMO_\theta}$$
for all $f\in BMO_\theta(\mathbb R^n)$,  $k\geq 1$ and for all balls $B= B(x_0,r)\subset \mathbb R^n$.
\item There exists a constant $C= C(q,\theta)>0$ such that
$$\Big( \frac{1}{|2^k B|} \int_{2^k B} |f(y) - f_B|^q \Big)^{1/q}\leq C \frac{k (1+ 2^k r)^{2\theta}}{\log\Big(e+ \frac{1}{2^k r}\Big)} \|f\|_{LMO_\theta}$$
for all $f\in LMO_\theta(\mathbb R^n)$,  $k\geq 1$ and for all balls $B= B(x_0,r)\subset \mathbb R^n$.
\end{enumerate}

\end{Lemma}

\begin{Lemma}\label{the key lemma}
Let $1<q < \infty$ and $T\in \mathscr L^m_{\rho,\delta}$ with $0<\rho\leq 1$, $0\leq \delta<1, -n -1 < m\leq -(n+1) (1-\rho)$. Then, for each $N>0$, there exists $C= C(N) >0$ such that
$$\|Ta\|_{L^q(2^{k+1}B\setminus 2^k B)}\leq C \frac{2^{-c k}}{(1+ 2^k r)^N} |2^k B|^{1/q-1}$$
holds for all $(h^1,q)$-atom $a$ related to the ball $B= B(x_0,r)$ and for all $k =1,2,3,\ldots$, where $c= \min\{1, \frac{1+n+m}{\rho}\}$.
\end{Lemma}

\begin{Lemma}\label{the key lemma added}
Let $T\in \mathscr L^m_{\rho,\delta}$ with $0<\rho\leq 1$, $0\leq \delta<1, -n -1 < m\leq -(n+1) (1-\rho)$. Then the following two statements hold:
\begin{enumerate}[(i)]
\item If $b\in BMO_\theta(\mathbb R^n)$ for some $\theta\in [0,\infty)$, then there exists a constant $C>0$ such that for every $(h^1,2)$-atom $a$ related to the ball $B=B(x_0,r)$,
$$\|(b-b_B)Ta\|_{L^1}\leq C\|b\|_{BMO_\theta}.$$
\item If $b\in LMO_\theta(\mathbb R^n)$ for some $\theta\in [0,\infty)$, then there exists a constant $C>0$ such that for every $(h^1,2)$-atom $a$ related to the ball $B=B(x_0,r)$,
$$\log(e+ 1/r)\|(b-b_B)Ta\|_{L^1}\leq C\|b\|_{LMO_\theta}.$$
\end{enumerate}
\end{Lemma}

The proof of Lemma \ref{Ky2, Lemma 5.3} can be found in \cite[Lemmas 5.3 and 6.6]{Ky2} as the special cases. Now let us give the proofs for Lemmas \ref{the key lemma} and \ref{the key lemma added}.

\begin{proof}[Proof of Lemma \ref{the key lemma}]
If $1<r\leq 2$, then for every $x\in 2^{k+1}B\setminus 2^k B$ and $y\in B= B(x_0,r)$, we have $|x-y|\geq |x-x_0| - |y-x_0|\geq 2^k r - r\geq 1$. Hence, by (i) of Proposition \ref{AH, Theorem 1.1} and the H\"older inequality,
\begin{eqnarray*}
|Ta(x)| = \left| \int_{\mathbb R^n} K(x,y) a(y) dy\right| &\leq& \int_B |K(x,y)| |a(y)| dy\\
&\leq& C \int_B \frac{1}{|x-y|^{N+n+1}} |a(y)| dy\\
&\leq& C \frac{1}{|x-x_0|^{N+n +1}} \|a\|_{L^q} |B|^{1-1/q}\\
&\leq& C \frac{1}{(2^k r)^{N+n+1}}
\end{eqnarray*}
for all $x\in 2^{k+1}B\setminus 2^k B$. This implies that
\begin{eqnarray*}
\|Ta\|_{L^q(2^{k+1}B\setminus 2^k B)} &\leq& C \frac{1}{(2^k r)^{N+n+1}} |2^{k+1}B\setminus 2^k B|^{1/q}\\
&\leq& C \frac{1}{2^k r}\frac{1}{(1+ 2^k r)^N} |2^k B|^{1/q-1}\\
&\leq& C \frac{2^{-ck}}{(1+ 2^k r)^N} |2^k B|^{1/q-1}.
\end{eqnarray*}

In the case of $0<r\leq 1$, we have $\int_B a(y)dy=0$. Thus, for every $x\in 2^{k+1}B\setminus 2^k B$, from $1+n+m>0$, Proposition \ref{AH, Theorem 1.1}(ii) yields
\begin{eqnarray}\label{a point estimate}
|Ta(x)| = \left| \int_{\mathbb R^n} K(x,y) a(y) dy\right| &\leq& \int_B |K(x,y) - K(x,x_0)| |a(y)| dy\nonumber\\
&\leq& C \int_B \frac{|y-x_0|}{|x-x_0|^{\frac{1+n+m}{\rho}}} |a(y)| dy\nonumber\\
&\leq& C \frac{r}{(2^k r)^{\frac{1+n+m}{\rho}}},
\end{eqnarray}
where we used the fact that $|x-\xi|\sim |x-x_0|$ if $\xi\in B$. Let us now consider the following two cases:
\begin{enumerate}[(a)]
\item If $(2^k -1) r\geq 1$, then, by using Proposition \ref{AH, Theorem 1.1}(i), it is similar to the case $1<r\leq 2$ that for every $x\in 2^{k+1}B\setminus 2^k B$,
\begin{eqnarray*}
|Ta(x)| &\leq& C \frac{1}{(2^k r)^{N + n +\frac{1+n+m}{\rho}}}\\
&\leq& C \frac{2^{-ck}}{(2^k r)^{N+n}}.
\end{eqnarray*}
Therefore,
\begin{eqnarray*}
\|Ta\|_{L^q(2^{k+1}B\setminus 2^k B)} &\leq& C \frac{2^{-ck}}{(2^k r)^{N+n}} |2^{k+1}B\setminus 2^k B|^{1/q}\\
&\leq& C \frac{2^{-ck}}{(1+ 2^k r)^N} |2^k B|^{1/q-1}.
\end{eqnarray*}

\item If $(2^k -1) r< 1$, then since $m\leq -(n+1)(1-\rho)$, (\ref{a point estimate}) yields
\begin{eqnarray*}
\|Ta\|_{L^q(2^{k+1}B\setminus 2^k B)} &\leq& C \frac{r}{(2^k r)^{\frac{1+n+m}{\rho}}} |2^{k+1}B\setminus 2^k B|^{1/q}\\
&\leq& C \frac{1}{2^k} \frac{1}{(2^k r)^n} |2^k B|^{1/q}\\
&\leq& C \frac{2^{-ck}}{(1+ 2^k r)^N} |2^k B|^{1/q-1},
\end{eqnarray*}
which ends the proof of Lemma \ref{the key lemma}.
\end{enumerate}

\end{proof}

\begin{proof}[Proof of Lemma \ref{the key lemma added}]
(i) Since $r\leq 2$, by the H\"older inequality, the $L^2$-boundedness of $T$, Lemmas \ref{Ky2, Lemma 5.3}(i) and \ref{the key lemma}, we get
\begin{eqnarray*}
&& \|(b-b_B) Ta\|_{L^1}\\
&=&  \|(b-b_B) Ta\|_{L^1(2B)} +\sum_{k=1}^\infty \|(b-b_B) Ta\|_{L^1(2^{k+1}B \setminus 2^k B)}\\
&\leq&  \|b-b_B\|_{L^2(2B)} \|Ta\|_{L^2(2B)} + \sum_{k=1}^\infty \|b-b_B\|_{L^2(2^{k+1}B \setminus 2^k B)} \|Ta\|_{L^2(2^{k+1}B \setminus 2^k B)}\\
&\leq& C  |2B|^{1/2}\|b\|_{BMO_\theta} \|a\|_{L^2} +\\
&& + C \sum_{k=1}^\infty (k+1)(1 + 2^{k+1}r)^{2\theta}|2^{k+1}B|^{1/2}\|b\|_{BMO_\theta} \frac{2^{-ck}}{(1+ 2^k r)^{2\theta}}|2^k B|^{-1/2}\\
&\leq& C \|b\|_{BMO_\theta} + C \sum_{k=1}^\infty k 2^{-c k} \|b\|_{BMO_\theta}\\
&\leq& C \|b\|_{BMO_\theta},
\end{eqnarray*}
where $c= \min\{1, \frac{1+n+m}{\rho}\}>0$.

(ii) Setting $\varepsilon = c/2$ with $c= \min\{1, \frac{1+n+m}{\rho}\}>0$, it is easy to check that there exists a positive constant $C=C(\varepsilon)$ such that
$$\log(e+ kt)\leq C k^{\varepsilon} \log(e + t)$$
for all $k\geq 1, t>0$. As a consequence, we get 
$$
\log\Big( e+ \frac{1}{r} \Big) \leq C 2^{\varepsilon k} \log\Big( e+ \frac{1}{2^k r} \Big)
$$
for all $k\geq 1$. This, together with the H\"older inequality, Lemmas \ref{Ky2, Lemma 5.3}(i) and \ref{the key lemma}, gives
\begin{eqnarray*}
&&\log(e + 1/r) \|(b-b_B) Ta\|_{L^1}\\
&=& \log(e + 1/r) \|(b-b_B) Ta\|_{L^1(2B)} +\sum_{k=1}^\infty \log(e + 1/r) \|(b-b_B) Ta\|_{L^1(2^{k+1}B \setminus 2^k B)}\\
&\leq& \log(e + 1/r) \|b-b_B\|_{L^2(2B)} \|Ta\|_{L^2(2B)} +\\
&& + \sum_{k=1}^\infty \log(e + 1/r) \|b-b_B\|_{L^2(2^{k+1}B \setminus 2^k B)} \|Ta\|_{L^2(2^{k+1}B \setminus 2^k B)}\\
&\leq& C \log(e + 1/r) \frac{|2B|^{1/2}}{\log(e + 1/(2r))}\|b\|_{LMO_\theta} \|a\|_{L^2} +\\
&& + C \sum_{k=1}^\infty 2^{\varepsilon k} \log\Big( e+ \frac{1}{2^k r} \Big)\frac{(k+1)(1 + 2^{k+1}r)^{2\theta}}{\log\Big(e+ \frac{1}{2^{k+1}r}\Big)}|2^{k+1}B|^{1/2}\|b\|_{LMO_\theta} \frac{2^{-ck}}{(1+ 2^k r)^{2\theta}}|2^k B|^{-1/2}\\
&\leq& C \|b\|_{LMO_\theta} + C \sum_{k=1}^\infty k 2^{-\varepsilon k} \|b\|_{LMO_\theta}\\
&\leq& C \|b\|_{LMO_\theta},
\end{eqnarray*}
where we used the facts that $r\leq 2$ and $c=2\varepsilon$.

\end{proof}

We are now ready to prove the main theorem.

\begin{proof}[\bf Proof of Theorem \ref{the main theorem}] 
 (i) Assume that $b\in LMO_\theta(\mathbb R^n)$ for some $\theta\in [0,\infty)$. By Proposition \ref{boundedness via atoms}, it is sufficient to show that
$$
\|[b,T](a)\|_{h^1}\leq C \|b\|_{LMO_\theta}
$$
holds for all $(h^1,2)$-atoms $a$ related to the ball $B=B(x_0,r)$. To this ends, by Theorem C, we need to prove that
\begin{equation}\label{proof of the main theorem 1}
\|(b-b_B)a\|_{h^1}\leq C \|b\|_{LMO_\theta}
\end{equation}
and
\begin{equation}\label{proof of the main theorem 2}
\|(b-b_B)Ta\|_{h^1}\leq C \|b\|_{LMO_\theta}.
\end{equation}
Thanks to Theorem B, to establish (\ref{proof of the main theorem 1}) and (\ref{proof of the main theorem 2}), it is sufficient to prove that
$$\|f(b- b_B)a\|_{L^1}\leq C \|b\|_{LMO_\theta} \|f\|_{bmo}$$
and
$$\|f (b- b_B)Ta \|_{L^1}\leq C \|b\|_{LMO_\theta}\|f\|_{bmo}$$
for all $f\in C^\infty_c(\mathbb R^n)$. Indeed, since $f\in C^\infty_c(\mathbb R^n)$, it is well-known that $|f_B|\leq C \log(e+ 1/r) \|f\|_{bmo}$. Therefore, by the H\"older inequality and Lemma \ref{Ky2, Lemma 5.3}(ii), 
\begin{eqnarray*}
&&\|f (b- b_B)a\|_{L^1}\\
& \leq& \|(f-f_B)(b-b_B)a\|_{L^1} + \log(e+ 1/r) \|f\|_{bmo} \|(b-b_B)a\|_{L^1}\\
&\leq& \|(f-f_B)\chi_B\|_{L^4} \|(b-b_B)\chi_B\|_{L^4}\|a\|_{L^2}  +  \log(e+ 1/r) \|f\|_{bmo} \|(b-b_B)\chi_B\|_{L^2}\|a\|_{L^2}\\
&\leq& C |B|^{1/4} \|f\|_{BMO} |B|^{1/4} \|b\|_{LMO_\theta} |B|^{-1/2} + C \|f\|_{bmo} |B|^{1/2} \|b\|_{LMO_\theta} |B|^{-1/2} \\
&\leq& C \|b\|_{LMO_\theta} \|f\|_{bmo},
\end{eqnarray*}
where we used the facts that supp $a\subset B$ and $r\leq 2$.

By the H\"older inequality, the $L^2$-boundedness of $T$ and Lemmas \ref{Ky2, Lemma 5.3}(ii) and \ref{the key lemma},
\begin{eqnarray*}\label{proof of the main theorem 3}
&&\|(f-f_B)(b- b_B) Ta\|_{L^1} \\
&=& \|(f-f_B)(b- b_B) Ta\|_{L^1(2B)} + \sum_{k=1}^\infty \|(f-f_B)(b- b_B) Ta\|_{L^1(2^{k+1}B \setminus 2^k B)}\\
&\leq& \|f-f_B\|_{L^4(2B)}\|b-b_B\|_{L^4(2B)} \|Ta\|_{L^2} +\\
&& + \sum_{k=1}^\infty \|f-f_B\|_{L^4(2^{k+1}B \setminus 2^k B)}\|b-b_B\|_{L^4(2^{k+1}B \setminus 2^k B)} \|Ta\|_{L^2(2^{k+1}B\setminus 2^k B)}\\
&\leq& C |2B|^{1/4}\|f\|_{BMO} |2B|^{1/4}\|b\|_{LMO_\theta} \|a\|_{L^2}\\
&& + C \sum_{k=1}^\infty (k+1)|2^{k+1} B|^{1/4}\|f\|_{BMO} \frac{(k+1)(1+ 2^{k+1}r)^{2\theta}}{\log(e+ \frac{1}{2^{k+1}r})}|2^{k+1} B|^{1/4}\|b\|_{LMO_\theta}\frac{2^{-ck}}{(1+ 2^k r)^{2\theta}}|2^k B|^{-1/2}\\
&\leq& C \|f\|_{BMO}\|b\|_{LMO_\theta},
\end{eqnarray*}
where we used the facts that $r\leq 2$ and $c=\min\{1,\frac{1+n+m}{\rho}\}>0$. Combining this with (ii) of Lemma \ref{the key lemma added} allow to conclude that
\begin{eqnarray*}
\|f(b-b_B)Ta\|_{L^1}&\leq& \|(f-f_B)(b-b_B)Ta\|_{L^1} + |f_B|\|(b-b_B)Ta\|_{L^1}\\
&\leq& C \|b\|_{LMO_\theta} \|f\|_{BMO} + C \log(e+1/r) \|f\|_{bmo}\|(b-b_B) Ta\|_{L^1}\\
&\leq& C \|b\|_{LMO_\theta} \|f\|_{bmo},
\end{eqnarray*}
which completes the proof of (i).

(ii) By a symbol calculation (cf. \cite[Proposition 0.3.B]{Tay}), there exists $\sigma^*\in S^m_{\rho,\delta}$ such that $T$ is the conjugate operator of $T_{\sigma^*}$ whose symbol is $\sigma^*$. So (ii) can be viewed as a consequence of (i). This ends the proof of Theorem \ref{the main theorem}.

\end{proof}

\section{Appendix}

The following theorem yields the converse of Theorem \ref{the main theorem}. Although, it can be followed from Theorem 1.2 of Yang, Wang and Chen \cite{YWC}, however we also would like to give a proof here for completeness. Also, it should be pointed out that our approach is different from that of Yang, Wang and Chen.

\begin{theorem}
Let $b$ be a function in $BMO_\infty(\bR^n)$. Suppose that $[b,T]$ is bounded on $h^1(\bR^n)$ for all $T\in \mathscr L^{m}_{\rho,\delta}$ with $0\leq \delta<1, 0< \rho\leq 1, \delta \leq \rho$ and $-(n+1)< m \leq - (n+1)(1-\rho)$. Then, $b\in LMO_\infty(\bR^n)$.
\end{theorem}

\begin{proof}
Assume that $b$ is a function in $BMO_\theta(\mathbb R^n)$, for some $\theta\in [0,\infty)$, such that $[b,T]$ is bounded on $h^1(\mathbb R^n)$ for all $T\in \mathscr L^m_{\rho,\delta}$ with  $0\leq \delta<1, 0< \rho\leq 1, \delta \leq \rho$ and $-(n+1)< m \leq - (n+1)(1-\rho)$. Then, for any $r_j, j=1,2,\ldots,n$, the classical local Riesz transform of Goldberg (see \cite{Go} for details), the commutator $[b,r_j]$ is bounded on $h^1(\mathbb R^n)$ since $r_j\in \mathscr L^0_{1,0}$ (e.g. \cite{HK}). Therefore, for every $(h^1,2)$-atom $a$ related to the ball $B$, (i) of Lemma \ref{the key lemma added} yields
\begin{eqnarray*}
\|r_j((b-b_B)a)\|_{L^1} &\leq& \|(b-b_B)r_j\|_{L^1} + C \|[b,r_j](a)\|_{h^1}\\
&\leq& C \|b\|_{BMO_\theta} + C \|[b,r_j]\|_{h^1\to h^1}.
\end{eqnarray*}
By the local Riesz transforms characterization (see \cite[Theorem 2]{Go}), we get
\begin{equation}\label{converse for local hardy spaces}
\|(b-b_B)a\|_{h^1} \leq C \left(\|b\|_{BMO_\theta} + \sum_{j=1}^n \|[b,r_j]\|_{h^1\to h^1}\right),
\end{equation}
for all $(h^1,2)$-atom $a$ related to the ball $B$, where the constant $C$ is independent of $b$ and $a$. We now prove that $b\in LMO_\theta(\mathbb R^n)$. To do this, since $b\in BMO_\theta(\mathbb R^n)$, it is sufficient to show that
$$\frac{\log(e+1/r)}{(1+r)^\theta}\frac{1}{|B|}\int_B |b(x)- b_B| dx\leq C \left(\|b\|_{BMO_\theta} + \sum_{j=1}^n \|[b,r_j]\|_{h^1\to h^1}\right)$$
holds for all $B= B(x_0,r)$ the ball in $\mathbb R^n$ satisfying $0<r<1/2$. Indeed, let $f$ be the signum function of $b-b_B$ and $a= (2|B|)^{-1} (f- f_B)\chi_B$. Then it is easy to see that $a$ is an $(h^1,2)$-atom related to the ball $B$. We next consider the function
$$g_{x_0,r}(x) = \chi_{[0,r]}(|x- x_0|)\log(1/r) + \chi_{(r,1]}(|x-x_0|)\log(1/|x-x_0|).$$
Then, thanks to \cite[Lemma 2.5]{MSTZ}, we have $\|g_{x_0,r}\|_{bmo}\leq C$. Moreover, it is clear that $g_{x_0,r}(b-b_B)a\in L^1(\mathbb R^n)$. By (\ref{converse for local hardy spaces}) and $bmo(\mathbb R^n)=(h^1(\mathbb R^n))^*$,
\begin{eqnarray*}
\frac{\log(e+1/r)}{(1+r)^\theta}\frac{1}{|B|}\int_B |b(x)- b_B| dx &\leq& 3 \log(1/r) \frac{1}{|B|}\int_B |b(x)- b_B| dx\\
&=& 6 \left| \int_{\mathbb R^n} g_{x_0,r}(x)(b(x)-b_B)a(x)dx \right|\\
&\leq& C \|g_{x_0,r}\|_{bmo} \|(b-b_B)a\|_{h^1}\\
&\leq& C \left(\|b\|_{BMO_\theta} + \sum_{j=1}^n \|[b,r_j]\|_{h^1\to h^1}\right).
\end{eqnarray*} 
This proves that $b\in LMO_\theta(\mathbb R^n)$, moreover,
$$\|b\|_{LMO_\theta}\leq C \left(\|b\|_{BMO_\theta} + \sum_{j=1}^n \|[b,r_j]\|_{h^1\to h^1}\right).$$

\end{proof}

Let $b\in L^1_{\rm loc}(\bR^n)$. A function $a$ is called an $h^1_b$-atom related to the ball $B=B(x_0,r)$ if $a$ is a $(h^1,\infty)$-atom related to the ball $B=B(x_0,r)$, and when $0<r<1$, it also satisfies $\int_{\mathbb R^n} a(x)b(x) dx =0$.

We define $h^1_b(\bR^n)$ as the space of finite linear combinations of $h^1_b$-atoms. As usual, the norm on $h^1_b(\bR^n)$ is defined by
$$\|f\|_{h^1_b}=\inf\left\{ \sum_{j=1}^N \lambda_j a_j: f= \sum_{j=1}^N \lambda_j a_j \right\}.$$

Given $b\in BMO_{\infty}(\bR^n)$, similar to a result of P\'erez \cite[Theorem 1.4]{Pe}, we find a subspace of $h^1(\bR^n)$ for which $[b,T]$ is bounded from this space into $L^1(\bR^n)$. In particular, we have:

\begin{theorem}\label{L^1 estimate for the commutator}
Let $b\in BMO_\infty(\bR^n)$ and $T\in \mathscr L^m_{\rho,\delta}$  with  $0\leq \delta<1, 0< \rho\leq 1, \delta \leq \rho$ and $-(n+1)< m \leq - (n+1)(1-\rho)$. Then, $[b,T]$ is bounded from $h^1_b(\bR^n)$ into $L^1(\bR^n)$.
\end{theorem}

\begin{proof}
Assume that $b\in BMO_\theta(\mathbb R^n)$ for some $\theta\in [0,\infty)$. It is sufficient to prove that for all $h^1_b$-atom $a$ related to the ball $B=B(x_0,r)$,
\begin{equation}
\|[b,T](a)\|_{L^1}\leq C \|b\|_{BMO_\theta}.
\end{equation}
Indeed, we first remark that supp $((b-b_B)a)\subset B$ and $\|(b-b_B)a\|_{L^2}\leq C \|b\|_{BMO_\theta}|B|^{1/2}$ by (i) of Lemma \ref{Ky2, Lemma 5.3}. Moreover, if $0<r<1$, then $\int_{\bR^n} (b(x)-b_B)a(x)dx = \int_{\bR^n} a(x)b(x)dx - b_B \int_{\bR^n} a(x)dx=0$. Therefore, $(b-b_B)a$ is a multiple of an $(h^1,2)$-atom. So, by (i) of Lemma \ref{the key lemma added} and Theorem C, we get
\begin{eqnarray*}
\|[b,T](a)\|_{L^1} &\leq& \|(b-b_B)Ta\|_{L^1} + \|T((b-b_B)a)\|_{L^1}\\
&\leq& C\|b\|_{BMO_\theta},
\end{eqnarray*}
which ends the proof of Theorem \ref{L^1 estimate for the commutator}.

\end{proof}


\end{document}